\documentclass[preprint,12pt]{elsarticle}

\usepackage{amssymb}
\usepackage{amsmath}
\usepackage{bm}
\usepackage{amsthm}

\newcommand{\adots}{
  \mathinner{\mkern1mu\raise1pt\hbox{.}\mkern2mu\raise4pt\hbox{.}
  \mkern2mu\raise7pt\vbox{\kern7pt\hbox{.}}\mkern1mu}}

\newcommand{\sign}{\mathop{\mathrm{sign}}}

\renewcommand{\Re}{\mathop{\mathrm{Re}}}
\renewcommand{\Im}{\mathop{\mathrm{Im}}}

\newcommand{\range}{\mathop{\mathrm{range}}}

\usepackage{bm}
\usepackage{mathtools}
\DeclareMathAlphabet{\matheul}{U}{eus}{m}{n}

\theoremstyle{plain}
\newtheorem{thm}{Theorem}[section]

\newtheorem{pro}[thm]{Proposition}
\theoremstyle{definition}
\newtheorem{defn}[thm]{Definition}
\newtheorem{exmp}{Example}[section]
\journal{Linear Algebra and its Applications}

\begin{document}

\begin{frontmatter}



\title{A New  Polar Decomposition \\ in a Scalar Product Space}


\author{Xuefang Sui\fnref{e1}}
\author{Paolo Gondolo\fnref{e2}}

\fntext[e1]{{\it Email address:} xuefang.sui@utah.edu}
\fntext[e2]{{\it Email address:} paolo.gondolo@utah.edu}

\address{Department of Physics and Astronomy, University of Utah, 115 South 1400 East \#201, Salt Lake City, UT 84112-0830}

\begin{abstract}
There are various definitions of right and left polar decompositions of an $m\times n$ matrix $F \in \mathbb{K}^{m\times n}$ (where $\mathbb{K}=\mathbb{C}$ or $\mathbb{R}$) with respect to bilinear or sesquilinear products defined by nonsingular matrices $M\in \mathbb{K}^{m\times m}$ and $N\in \mathbb{K}^{n\times n}$.  The existence and uniqueness of such decompositions under various assumptions on $F$, $M$, and $N$ have been studied.  Here we  introduce a new form of right and left polar decompositions, $F=WS$ and $F=S'W'$, respectively, where the matrix $W$ has orthonormal columns ($W'$ has orthonormal rows) with respect to suitably defined scalar products which are functions of $M$, $N$, and $F$, and the matrix $S$ is selfadjoint with respect to the same suitably defined scalar products and has eigenvalues only in the open right half-plane. We show that our right and left decompositions exist and are unique for any nonsingular matrices $M$ and $N$ when the matrix $F$ satisfies $(F^{[M,N]})^{[N,M]}=F$ and $F^{[M,N]}F$ ($FF^{[M,N]}$, respectively) is nonsingular, where  $F^{[M,N]}=N^{-1} F^\# M$ with $F^\#=F^T$  for real or complex bilinear forms and $F^\#=\bar{F}^T$ for sesquilinear forms. When $M=N$, our results apply to nonsingular square matrices $F$. Our assumptions on $F$, $M$, and $N$ are in some respects weaker and in some respects stronger than those of previous work on polar decompositions.  
\end{abstract}


\end{frontmatter}

% \linenumbers


\section{Introduction}

Let $\mathbb{K}$ denote either  the field of real numbers $\mathbb{K} = \mathbb{R}$ or the field of complex numbers $\mathbb{K} = \mathbb{C}$.
Consider a  scalar product $[x,y]_N$ over the field  $\mathbb{K}$,  i.e.,   a bilinear or sesquilinear form defined by a nonsingular matrix $N\in \mathbb{K}^{n \times n}$:
\begin{align}
[x,y]_N=\begin{cases}
x^T Ny,   \quad \text{for real or complex bilinear forms}, \\
\bar{x}^TNy ,  \quad \text{for sesquilinear forms},
\end{cases}
\end{align}
for $x, y \in \mathbb{K}^n$, where $x^T$ indicates the transpose of $x$ and $\bar{x}$ indicates the complex conjugate of $x$.   

The $N$-adjoint $A^{[N]}$ of a matrix $A\in \mathbb{K}^{n \times n}$ is defined as $[Ax,y]_N=[x,A^{[N]}y]_N$ for all $x,y \in \mathbb{K}^n$,  that is 
\begin{align}
A^{[N]}=N^{-1}A^\# N
\end{align}
with (using the notation in \cite{Higham2010})
\begin{align}
A^\#=\begin{cases}
A^T,  \quad \text{for real or complex bilinear forms}, \\
\bar{A}^T,   \quad \text{for sesquilinear forms}.
\end{cases}
\end{align}

A matrix  $A \in  \mathbb{K}^{n\times n}$ is called selfadjoint if $A^\#=A$ and unitary if $A^\# A=I_n$, where $I_n\in\mathbb{K}^{n\times n}$ is the identity matrix.
A matrix  $A \in  \mathbb{K}^{n\times n}$ is called $N$-selfadjoint if $A^{[N]}=A$, or $A^\# N=NA$, and $N$-unitary if $A^{[N]} A=I_n$, or $A^\# N A = N$.

The $(M,N)$-adjoint  $A^{[M, N]}\in \mathbb{K}^{n\times m}$ of a matrix $A\in \mathbb{K}^{m\times n}$, where $M \in \mathbb{K}^{m \times m}$ and $N\in \mathbb{K}^{n\times n}$ are nonsingular matrices, is defined by the condition  (see e.g.~\cite{Higham2010})
\begin{align}
[Ax, y]_M=[x,A^{[M, N]} y]_N  \qquad \text{for all} \quad x\in \mathbb{K}^n,y\in \mathbb{K}^m.
\end{align}
The $(M,N)$-adjoint $A^{[M, N]}$ can be written as
\begin{align}
A^{[M, N]}=N^{-1} A^\# M.
\end{align}
If $M=N$, then $A^{[N,N]}=A^{[N]}$. Two useful properties of the $(M,N)$-adjoint are\\ (a) for $A\in \mathbb{K}^{m\times n}$ and $B\in \mathbb{K}^{n\times n}$,
\begin{align}
(AB)^{[M,N]}=N^{-1}B^\# A^\# M=N^{-1}B^\#  N N^{-1}A^\# M=B^{[N]}A^{[M,N]};
\end{align}
 (b) for $A\in \mathbb{K}^{m\times m}$ and $B\in \mathbb{K}^{m\times n}$,
\begin{align}
(AB)^{[M,N]}=N^{-1}B^\# A^\# M=N^{-1}B^\# MM^{-1}A^\# M=B^{[M,N]}A^{[M]}.
\end{align}

The classical concept of polar decomposition in  a Euclidean space is expressed by the following statements.  
Any matrix $F \in \mathbb{K}^{m\times n} $ with $m\ge n$ has a right polar decomposition 
\begin{align}
F = U S,
\end{align} 
where  the matrix $ U \in  \mathbb{K}^{m\times n} $ has orthonormal columns (i.e., $U^\ast U = I_n$) and  the matrix $ S  \in \mathbb{K}^{n\times n} $, uniquely given by $S=\sqrt{F^\ast F}$, is selfadjoint and positive-semidefinite (i.e., all of its eigenvalues are real and positive or zero). Here $A^\ast=A^T$ for $\mathbb{K}=\mathbb{R}$  and  $A^\ast=\bar{A}^T$ for $\mathbb{K}=\mathbb{C}$.   $\sqrt{A}$ is  the principal square root of a positive-definite matrix $A$ or the unique  positive-semidefinite  square root of a positive-semidefinite  selfadjoint  matrix $A$ with zero eigenvalues.
Similarly, any matrix $F \in \mathbb{K}^{m\times n} $ with $m\le n$ has a left polar decomposition 
\begin{align}
F = S' \, U',
\end{align} 
where  the matrix $ U' \in  \mathbb{K}^{m\times n} $ has orthonormal rows (i.e., $U' U^{\prime\ast} = I_m$) and  the matrix $ S'  \in \mathbb{K}^{m\times m} $, uniquely given by $S'=\sqrt{F F^\ast}$, is selfadjoint and positive-semi\-definite.  In  case $F\in\mathbb{K}^{n\times n}$ is a square nonsingular matrix, the right and left polar decompositions are unique with $U=FS^{-1}$ and $U'=S^{\prime-1}F$, respectively, and $U$ and $U'$ are unitary matrices.

There is a long history on generalizing the classical polar decomposition to scalar products on $\mathbb{K}^n$ given by a bilinear or sesquilinear form.  Different generalizations exist based on the dimension of the matrices involved and the required properties of the factors in the decomposition.

To the extent of our knowledge, polar decompositions of square matrices $F\in\mathbb{K}^{n\times n}$ with respect to scalar products defined by a nonsingular matrix $N\in\mathbb{K}^{n\times n}$ are of the form
\begin{align}
\label{eq:Hpolar}
F = L S ,
\end{align}
where $L\in\mathbb{K}^{n\times n}$ is $N$-unitary and $S\in\mathbb{K}^{n\times n}$ is $N$-selfadjoint. For example, in~\cite{Bolshakov1995, Bolshakov1996, Bolshakov19961997, Bolshakov1997,  Bolshakov19971997, Kintzel2005, Mehl2005}, necessary and sufficient conditions are given for the existence of decompositions of the form (\ref{eq:Hpolar}) when the matrix $N$ is a selfadjoint matrix $H$ (such decompositions are called $H$-polar decompositions and are not necessarily unique even if $F$ is nonsingular). Further restricting the $H$-selfadjoint factor to satisfy the condition that $HS$ is positive-semidefinite, still with $H$ selfadjoint, leads to the semidefinite $H$-polar decomposition of a square matrix $F$ in~\cite{Bolshakov19971997, Kintzel2005}, which exists (but is in general not unique) if and only if $F^{[H]}F$ is diagonalizable and has only nonnegative real eigenvalues and $\ker(F)$ satisfies the restrictions in Theorem 5.3 of~\cite{Bolshakov19971997}. The generalized polar decomposition of a square matrix $F$ studied in~\cite{Cardoso2002, Higham2004, Higham2005, Higham2010, Mackey2006} is also of the form (\ref{eq:Hpolar}), but with the additional restriction that the nonzero eigenvalues of $S$ are contained in the open right half-plane, i.e., considering both zero and nonzero eigenvalues,  the spectrum of $S$: $\Lambda(S) \subseteq \{ z\in \mathbb{C} : \Re(z) > 0 \} \cup \{0\}$. In particular, \cite{Higham2005} shows that a generalized polar decomposition of a nonsingular square matrix $F$ exists if and only if $F^{[N]} F$ has no negative real eigenvalues and $(F^{[N]})^{[N]} = F$, and that when such a factorization exists, it is unique. And~\cite{Higham2010} shows that for an orthosymmetric scalar product $N$ (i.e., such that $(F^{[N]})^{[N]} = F$ for all $F\in\mathbb{K}^{n\times n}$) a generalized polar decomposition of singular and nonsingular square matrices $F$ exists with a unique $N$-selfadjoint factor $S$ if and only if $\ker(F^{[N]} F) = \ker(F)$, $F^{[N]}F$ has no negative real eigenvalues and the zero eigenvalues of $F^{[N]}F$, if any, are semisimple.

Polar decompositions of rectangular matrices $F\in\mathbb{K}^{m\times n}$ with respect to underlying scalar products defined by nonsingular matrices $M$ and $N$ on $\mathbb{K}^m$ and $\mathbb{K}^n$ have been studied in~\cite{Higham2010} (see also~\cite{Bolshakov1995, Kintzel2005} for the special case of two selfadjoint scalar products with $m=n$). In~\cite{Higham2010}, under the assumption that $M$ and $N$ form an orthosymmetric pair (see their Definition 3.2), a canonical generalized polar decomposition of $F\in\mathbb{K}^{m\times n}$ is defined as a decomposition
\begin{align}
F={\it\Lambda} S,
\end{align}
in which the matrix ${\it\Lambda}\in\mathbb{K}^{m\times n}$ is a partial $(M,N)$-isometry, i.e., ${\it\Lambda} {\it\Lambda}^{[M,N]} {\it\Lambda} = {\it\Lambda}$, the matrix $S\in\mathbb{K}^{n\times n}$ is $N$-selfadjoint with nonzero eigenvalues contained in the open right half-plane, and $\range({\it\Lambda}^{[M,N]})=\range(S)$. Theorem 3.9 in~\cite{Higham2010} states that, for an orthosymmetric pair of nonsingular matrices $M\in\mathbb{K}^{m\times m}$ and $N\in\mathbb{K}^{n\times n}$, a matrix $F\in\mathbb{K}^{m\times n}$ has a unique canonical polar decomposition if and only if $\ker(F^{[M,N]} F) = \ker(F)$, $F^{[M,N]}F$ has no negative real eigenvalues and the zero eigenvalues of $F^{[M,N]}F$, if any, are semisimple.

In this paper, we replace  the condition of orthosymmetric pair  of  $M$ and $N$ by $(F^{[M,N]})^{[N,M]}=F$,  and we drop the condition of no negative real eigenvalues of $F^{[M,N]}F$ (which is $F^{[N]}F$ for square matrices $F$ with $M=N$), by extending the form of the polar decomposition of a matrix $F\in\mathbb{K}^{m\times n}$ with $m\ge n$ to
\begin{align}
F = W S ,
\end{align}
where $W\in\mathbb{K}^{m\times n}$ has orthonormal columns with respect to suitably defined scalar products which are functions of $M$, $N$, and $F$ but is not necessarily $N$-unitary or a partial $(M,N)$-isometry, while $S\in\mathbb{K}^{n\times n}$ is selfadjoint with respect to the same suitably defined scalar products, has  eigenvalues only in the open right half-plane, and turns out to be unique. When $m\le n$ we use a left decomposition
\begin{align}
F = S' W' 
\end{align}
with analogous properties for $S'\in\mathbb{K}^{m\times m}$ and $W'\in\mathbb{K}^{m\times n}$ (in this case, $W'$ has orthonormal rows instead of orthonormal columns, with respect to the suitably defined scalar products). However, we do all this here only for matrices $F\in\mathbb{K}^{m\times n}$ for which $F^{[M,N]}F$ if $m\ge n$, or $FF^{[M,N]}$ if $m\le n$, is nonsingular. Thus our assumptions on $F$, $M$, and $N$ in this paper are in some respects weaker and in some respects stronger than those of previous work on polar decompositions. Our right and left decomposition Theorems~\ref{thm:F=WS_MN} and~\ref{thm:F=SW_MN} for matrices $F\in\mathbb{K}^{m\times n}$ closely resemble the classical right and left polar-decomposition theorems mentioned above.

In Section~\ref{sec:definitions}, we define $r$-positive-definite matrices and extend the concept of matrices with orthonormal columns and rows to generic pairs of scalar products. In Section~\ref{sec:function} we introduce generalized matrix sign functions. In Section~\ref{sec:F=WS}, we present our right and left decompositions for nonsingular square matrices with respect to one scalar product. In Section~\ref{sec:F=WS_MN},  we present our right and left decompositions for rectangular matrices  with respect to  two scalar products. Finally, in  Section~\ref{sec:exmp},  we make comments on our theorems and give some examples.

\section{Some useful definitions}
\label{sec:definitions}

In this section, we introduce some definitions to simplify the language.

First we extend the definition of positive-definite matrices to matrices with nonreal eigenvalues as follows. 
\begin{defn}($r$-positive-definite)
\label{defn:r-positive-definite}
A matrix $A\in\mathbb{K}^{n\times n}$ is $r$-positive-definite if all of its eigenvalues have positive real part,  or equivalently if all of its eigenvalues lie in the open right half-plane. 
\end{defn}  

Secondly, we extend the definitions of matrices with orthonormal rows or columns to the case of non-Euclidean scalar products.

\begin{defn}($(M,N)$-orthonormal columns, $(M,N)$-orthonormal rows, and $(M,N)$-unitarity)
\label{defn:m-n-orthonormal-columns}
Let $M \in \mathbb{K}^{m \times m}$ and $N\in \mathbb{K}^{n\times n}$ be nonsingular matrices. A matrix $W\in\mathbb{K}^{m\times n}$ has $(M,N)$-orthonormal columns if $W^{[M,N]} W =I_n$. 
A matrix $W\in\mathbb{K}^{m\times n}$ has $(M,N)$-orthonormal rows if $WW^{[M,N]} =I_m$. If a matrix $W\in\mathbb{K}^{m\times n}$ has both $(M,N)$-orthonormal rows and $(M,N)$-orthonormal columns (in which case $m=n$,  $W$ is a nonsingular square matrix),  we say that $W$ is $(M,N)$-unitary.
\end{defn}  

The condition that $W\in\mathbb{K}^{m\times n}$ has $(M,N)$-orthonormal columns can also be written as
\begin{align}
\label{eq:WNWN}
W^\# M W=N,
\end{align}
or
\begin{align}
\label{eq:WNWNbis}
[Wx,Wy]_{M}=[x,y]_{N} \quad  \text{for all} \quad x,y \in \mathbb{K}^n,
\end{align}
while the condition that $W\in\mathbb{K}^{m\times n}$ has $(M,N)$-orthonormal rows can also be written as
\begin{align}
M^{-1} = W N^{-1} W^\#,
\end{align}
or
\begin{align}
[x,y]_{M^{-1}} = [W^\# x,W^\# y]_{N^{-1}} \quad  \text{for all} \quad x,y \in \mathbb{K}^m.
\end{align}

When $m=n$, a square matrix $W\in\mathbb{K}^{n\times n}$ that has $(M,N)$-orthonormal columns or $(M,N)$-orthonormal rows is necessarily nonsingular.  A nonsingular matrix $W\in\mathbb{K}^{n\times n}$ has $(M,N)$-orthonormal columns if and only if it has $(M,N)$-orthonormal rows.  By Definition~\ref{defn:m-n-orthonormal-columns}, $W$ is $(M,N)$-unitary.     

A square matrix $W\in\mathbb{K}^{n\times n}$ that satisfies conditions (\ref{eq:WNWN}) and (\ref{eq:WNWNbis}) for nonsingular selfadjoint matrices $M$ and $N$ has been called $(N,M)$-unitary or $N$-$M$-unitary in~\cite{Bolshakov1995, Gohberg2005}.  Our Definition~\ref{defn:m-n-orthonormal-columns} extends these notions to more general matrices.

\section{Generalized matrix sign function}
\label{sec:function}
In this section, we introduce a  primary  matrix function called  a generalized matrix sign function. 

First, we recall  some facts about  primary matrix functions (see, e.g.,~\cite{Gantmacher1977, Higham2008, HornJohnson1991}).  A primary matrix function $f$ of a matrix $A\in\mathbb{K}^{n\times n}$ can be defined by means of a function $f: \mathbb{C} \to \mathbb{C}$ defined on the spectrum of $A$.  The  function $f: \mathbb{C} \to \mathbb{C}$ is  called the stem function of the matrix function $f$ and here they are denoted by the same letter. 

\begin{defn}(chapter~V in~\cite{Gantmacher1977} or Definition 1.1 in~\cite{Higham2008}) 
A  function $f: \mathbb{C} \to \mathbb{C}$ is said to be defined on the spectrum of a matrix $A\in\mathbb{K}^{n\times n}$ if its value $f(\lambda_k)$  and the values of its $s_k-1$ derivatives 
\begin{align}
\label{eq:derivative}
f^{(j)}(\lambda_k),\quad \quad \quad j=0, \hdots,  s_k-1, \quad \quad k =1, \hdots, t,
\end{align}
exist at all eigenvalues  $\lambda_k$ of $A$.  Here $s_k$  is the size of the Jordan blocks $J_{s_k}(\lambda_k)$ in the Jordan decomposition of $A$.
\end{defn} 
As remarked in~\cite{Higham2008} right after Definition 1.1, arbitrary numbers can be assigned as the values of $f(\lambda_k)$ and its derivatives $f^{(j)}(\lambda_k), j=1,\hdots, s_k-1$, at each eigenvalue $\lambda_k$ of $A$.

 With a stem function $f$,  the primary matrix function  $f(A)$ of $A\in\mathbb{K}^{n\times n}$ is given by  $f(A)=f(QJQ^{-1})=Qf(J)Q^{-1}$, where, with a nonsingular matrix $Q\in\mathbb{K}^{n\times n}$,  $A=QJQ^{-1}$ is a Jordan decomposition of $A$ and $f(J)$ is calculated by applying the stem function to each Jordan block.
A primary matrix function $f(A)$  of a matrix $A\in\mathbb{K}^{n\times n}$ is well defined in the sense that it is unique. 

By  Thm.~1.12 in~\cite{HornJohnson1991},  a  primary matrix function of $A\in\mathbb{K}^{n\times n}$  is  a polynomial in $A$.  It follows that all the primary matrix functions $f(A)$ commute with the matrix $A$  and also  commute with  each other. 

The relations $f(A^T) = f(A)^T$ and $f(Q^{-1} A Q) = Q^{-1} f(A) Q$ for a nonsingular matrix $Q\in\mathbb{K}^{n\times n}$ hold for any primary matrix function $f$ of a matrix $A\in\mathbb{K}^{n\times n}$.  If the stem function in Equation~(\ref{eq:derivative}) satisfies $f^{(j)}(\bar{\lambda})=\overline{f^{(j)}(\lambda)}$, then $f(\bar{A})=\overline{f(A)}$, and $f(A)$ is real when $A$ is real. It follows that, with a product matrix $N\in\mathbb{K}^{n\times n}$, for bilinear forms $f(A^{[N]})=f(A)^{[N]}$ always holds, while for sesquilinear forms,  $f(A^{[N]})=f(A)^{[N]}$ is equivalent to $f(\bar{A})=\overline{f(A)}$  (see Theorem 3.1~\cite{Higham2005}).  If  $f(A^{[N]})=f(A)^{[N]}$ and $A$ is $N$-selfadjoint, then both $f(A)$ and $f(A)A$ are $N$-selfadjoint. 

We now introduce the concept of a generalized matrix sign function.

\begin{defn}(generalized sign function)
\label{defn:sign_function}
A primary matrix function $\sigma: \mathbb{K}^{n\times n} \to \mathbb{K}^{n\times n}$ is called a generalized matrix sign function if for all nonsingular  matrices $A\in\mathbb{K}^{n\times n}$, $\sigma(A)A$ has no eigenvalues on the negative real axis,  $\sigma(\bar{A})=\overline{\sigma(A)}$, and all the eigenvalues $\lambda$ of $\sigma(A)$ are on the unit circle $|\lambda|=1$.
\end{defn}

Notice that the generalized sign $\sigma(A)$ of an $N$-selfadjoint matrix $A$ is $N$-selfadjoint.

We give some examples of generalized matrix sign functions.

In~\cite{Roberts1971}, a matrix sign function is defined as  the primary matrix function associated to the stem function
\begin{align}
\text{sign}(\lambda)=\begin{cases}
+1, \quad\quad & \text{for }  \Re \lambda>0, \\
-1, \quad\quad &\text{for }    \Re \lambda<0 , \\
\text{undefined},  \quad\quad &\text{for }    \Re \lambda=0.
\end{cases}
\end{align}
Also see,  e.g.,~\cite{Higham1994, KenneyLaub1995}.  It is shown in~\cite{Higham1994} that for  a matrix $A$ such that $A^2$ is nonsingular and has no negative real  eigenvalues,  the relation $\text{sign}(A)=A(A^2)^{-1/2}$ holds.

The first example of a generalized sign function is obtained by extending the stem function $\sign(\lambda)$ to the imaginary axis. For example, the function
\begin{align}
\label{eq:Sign_1}
f(\lambda)=\begin{cases}
+1, \quad\quad & \text{for }  \Re \lambda>0, \\
-1, \quad\quad &\text{for }    \Re \lambda<0 , \\
+1,  \quad\quad &\text{for }    \Re \lambda=0,  \Im \lambda \ne 0, \\
\text{undefined},\quad\quad &\text {for} \quad \lambda=0,
\end{cases}
\end{align}
with
\begin{align}
f^{(j)}(\lambda) = 0, \quad \text{for } j\ge 1,
\end{align}
is a stem function of a generalized matrix sign function.  
 
A second example of a generalized sign function arises from the stem function
\begin{align}
\label{eq:Sign}
f(\lambda)=\begin{cases}
\text{undefined},\quad\quad &\text {for} \quad \lambda=0,\\
-1, \quad \quad &\text{for}\quad \Re \lambda <0,  \Im \lambda=0, \\
+1, \quad\quad &\text {otherwise},
\end{cases}
\end{align}
with 
\begin{align}
f{^{(j)}}(\lambda)=0,  \quad  \text{for } j\ge 1.
\end{align}

A third example of a generalized sign function defined on the spectrum of a  matrix is
\begin{align}
\label{eq:Sign_3}
f(\lambda)=\begin{cases}
{\bar{\lambda}}/{|\lambda|}, \quad \quad &\text{for}  \quad \lambda \ne0, \\
\text{undefined},\quad\quad &\text {for} \quad \lambda=0,
\end{cases}
\end{align}
and  all the derivatives are defined to be zero, i.e., 
\begin{align}
f{^{(j)}}(\lambda)=0,  \quad  \text{for } j\ge 1.
\end{align}
Since in this paper we define the generalized sign function for nonsingular matrices, we leave $f(\lambda)$ at $\lambda=0$ undefined.

\section{Right and left decompositions of square matrices}
\label{sec:F=WS}

\begin{thm}
\label{thm:F=WS}
Given a nonsingular scalar product defined by  $N\in \mathbb{K}^{n\times n}$ and a generalized sign function $\sigma: \mathbb{K}^{n\times n} \to \mathbb{K}^{n\times n}$,  a matrix $F\in \mathbb{K}^{n\times n}$ has a  decomposition 
\begin{align}
\label{eq:F=WS}
F = W S,
\end{align}
where, with $\Sigma=\sigma(F^{[N]}F)$, the matrix  $W\in \mathbb{K}^{n\times n}$ is $(N, N\Sigma^{-1})$-unitary with $(W^{[N]})^{[N]}=W$ and  the matrix $S \in \mathbb{K}^{n\times n}$ is $r$-positive-definite, $N$-selfadjoint and $N\Sigma$-selfadjoint, if and only if $F$ is nonsingular and $(F^{[N]})^{[N]}=F$.  When such a decomposition  exists it is unique, with $S$ given by $S=(\Sigma F^{[N]}F)^{1/2}$ and $W=FS^{-1}$.
\end{thm}

\begin{proof}   IF: Since $(F^{[N]})^{[N]}=F$, then $F^{[N]} F$ is $N$-selfadjoint besides being nonsingular.  Since $\sigma$ is a generalized sign function,  $\Sigma=\sigma(F^{[N]}F)$ is $N$-selfadjoint and commutes with $F^{[N]} F$, and $\Sigma F^{[N]} F$  is $N$-selfadjoint with no zero or negative real eigenvalues. Thus the principal square root
\begin{align} 
S = \Big( \Sigma F^{[N]} F \Big)^{1/2}
\end{align}
is well-defined and is $r$-positive-definite and $N$-selfadjoint (the latter follows from Lemma 3.1 (b) in~\cite{Mackey2006}, for example). 
Since both $\Sigma$ and $S$  are primary matrix functions and thus polynomials in $F^{[N]} F$,  $\Sigma$ and $S$ commute.  We compute
\begin{align}
S^{[N\Sigma]}=(N\Sigma)^{-1}S^\# (N\Sigma)=\Sigma^{-1}S^{[N]} \Sigma=\Sigma^{-1}S \Sigma=S.
\label{eq:18}
\end{align}
That is, $S$ is  $N\Sigma$-selfadjoint. 

Let $W  = F S^{-1}$.  Then $F = W S$.   We need to show that $W$ is $(N,N\Sigma^{-1})$-unitary, i.e.\ $W^\# N W=N\Sigma^{-1}$. 
Substituting $W=FS^{-1}$  in $W^{[N]}W$,  and using $S^{[N]}=S$, $S^{-[N]} =S^{-1}$,  $F^{[N]}F = \Sigma^{-1} S^2$ and $\Sigma S = S \Sigma$, we find
\begin{align}
\label{eq:detA}
W^{[N]}W=S^{-[N]} F^{[N]}FS^{-1}=S^{-1}\Sigma ^{-1}S^2 S^{-1} =\Sigma^{-1},  
\end{align}
from which  $W^\# NW=N\Sigma^{-1}$.  Moreover,
\begin{align}
(W^{[N]})^{[N]}=((FS^{-1})^{[N]})^{[N]}=(F^{[N]})^{[N]}(S^{-[N]})^{[N]}=FS^{-1}=W.
\end{align}

Now we show that the decomposition $F=WS$ is unique.   Assume there is another decomposition $F=\tilde{W}\tilde{S}$ where, with $\Sigma=\sigma(F^{[N]}F)$, $\tilde{W}$ is $(N, N\Sigma^{-1})$-unitary, $(\tilde{W}^{[N]})^{[N]}=\tilde{W}$,  $\tilde{S}$ is $r$-positive-definite, $N$-selfadjoint and $N\Sigma$-selfadjoint. Since $\tilde{S}$ is both $N$-selfadjoint and $N\Sigma$-selfadjoint, then $\tilde{S}$ commutes with $\Sigma$.  Thus $F^{[N]}F=\tilde{S} \tilde{W}^{[N]} \tilde{W}\tilde{S}=\tilde{S} \Sigma^{-1} \tilde{S}=\Sigma^{-1}  \tilde{S}^{2}$, and  $\Sigma F^{[N]}F=\tilde{S}^{2}$. Since  $\tilde{S}$ is $r$-positive definite and a square root of $\Sigma F^{[N]}F$, so $\tilde{S}$ is the principal square root of $\Sigma F^{[N]}F$, which is unique. Hence $\tilde{S}=S$. The uniqueness of $W$ follows from $\tilde{W}=W=FS^{-1}$.

ONLY IF:  Assume $F=WS$ where, with $\Sigma=\sigma(F^{[N]}F)$, $W$ is $(N, N\Sigma^{-1})$-unitary, $(W^{[N]})^{[N]}=W$,  $S$ is $r$-positive-definite, $N$-selfadjoint and $N\Sigma$-selfadjoint. Then
\begin{align}
(F^{[N]})^{[N]}=((WS)^{[N]})^{[N]}=(W^{[N]})^{[N]}(S^{[N]})^{[N]}=WS=F.
\end{align}
Since $S$ is $N$-selfadjoint and $N\Sigma$-selfadjoint, then $S$ commutes with $\Sigma$, and it follows that $\Sigma F^{[N]}F=S^2$. Since $S$ is $r$-positive-definite, implying that it is nonsingular,  then $\Sigma F^{[N]}F$ is nonsingular,  it follows that  $F^{[N]}F$ and $F$ are nonsingular. 
\end{proof}

\begin{thm}
\label{thm:F=SW}
If  a nonsingular matrix $F\in \mathbb{K}^{n\times n}$ has a decomposition  $F=WS$ in Theorem~\ref{thm:F=WS},  then $F$  also has a decomposition  
\begin{align}
\label{eq:F=SW}
F=S^\prime \, W,
\end{align}
where, with $\Sigma^\prime= \sigma(FF^{[N]})$,  the matrix $S^\prime \in \mathbb{K}^{n\times n}$ is  $r$-positive-definite, $N$-selfadjoint and $N\Sigma^\prime$-selfadjoint,  and the matrix $W\in \mathbb{K}^{n\times n}$  is $(N\Sigma^\prime,N)$-unitary.  When such a decomposition  exists it is unique, with $S^\prime$ given by $S^\prime=(\Sigma^\prime FF^{[N]})^{1/2}$ and $W=S^{\prime-1}F$.  $W$ in~(\ref{eq:F=WS}) and~(\ref{eq:F=SW}) is the same one. 
\end{thm}

\begin{proof} Let $\Sigma= \sigma(F^{[N]}F)$ and $\Sigma^\prime= \sigma(FF^{[N]})$.  Since 
\begin{align}
FF^{[N]}=&F(F^{[N]}F)F^{-1}=WS \Sigma ^{-1} S^2 (WS)^{-1} \nonumber\\
=&W \Sigma ^{-1} S^2 W^{-1}=W(F^{[N]}F)W^{-1},
\end{align}
and $\sigma$ is a primary matrix function, then $\Sigma^\prime=W\Sigma W^{-1}$ and $\Sigma^\prime FF^{[N]}=W(\Sigma F^{[N]}F)W^{-1}$.   Let 
\begin{align}
S^\prime=(\Sigma^\prime FF^{[N]})^{1/2}=W (\Sigma F^{[N]}F)^{1/2}W^{-1}=WSW^{-1}.
\end{align}
Since $FF^{[N]}$ and $\Sigma^\prime FF^{[N]}$ are $N$-selfadjoint, then $S^\prime$ is $N$-selfadjoint and  $N\Sigma^\prime$-selfadjoint, following a computation similar to (\ref{eq:18}).     Then $F=WS=WSW^{-1} W=S^\prime W$.   And 
\begin{align}
\label{eq:detB}
WW^{[N]}=S^{\prime-1}F(S^{\prime-1}F)^{[N]}=S^{\prime-1}FF^{[N]}S^{\prime-1}=\Sigma^{\prime -1}.  
\end{align}
Since $W$ is nonsingular,  it follows that $W^\# N\Sigma^\prime W=N$,   that is,  $W$ is $(N\Sigma^\prime, N)$-unitary. 

The proof of uniqueness is similar to the one of Theorem~\ref{thm:F=WS}.
\end{proof}

If the scalar product matrix $N$ is orthosymmetric (Definition 2.4 in~\cite{Mackey2006}), then  the condition $(F^{[N]})^{[N]}=F$ is satisfied for all matrices $F$. If $N$ is not orthosymmetric, then the validity of the condition $(F^{[N]})^{[N]}=F$ depends on the specific matrices $N$ and $F$ and on the nature of the scalar product defined by $N$ (real bilinear, complex bilinear or sesquilinear). For a given matrix $N$, some matrices $F$ may satisfy the condition $(F^{[N]})^{[N]}=F$ and some may not. For example, for the sesquilinear form defined by $N=\begin{pmatrix} &1\\2i &\end{pmatrix}$, the matrix $F=\begin{pmatrix}1&\\&i\end{pmatrix}$ satisfies $(F^{[N]})^{[N]}=F$ and the matrix $F=\begin{pmatrix}&1\\i&\end{pmatrix}$ does not. Moreover, it is not enough to specify the matrices $N$ and $F$, but one must specify the kind of scalar product. For example, take $N=\begin{pmatrix}&1\\i&\end{pmatrix}$  and  $F=\begin{pmatrix}&1\\4i& \end{pmatrix}$: if $N$ defines a sesquilinear product, then $(F^{[N]})^{[N]}=F$, and if $N$ defines a complex bilinear product, then $(F^{[N]})^{[N]} \ne F$.

We end this section by discussing the decomposition in the terms of the matrix determinants. 
\begin{pro}
For the decompositions $F=WS=S'W$ in Theorems~\ref{thm:F=WS} and~\ref{thm:F=SW}, the following holds. 
\begin{align}
& \det W = \pm 1 && \text{for real bilinear forms},
\label{eq:det1}
\\
& | \det W |  = 1 &&  \text{for complex bilinear  and sesquilinear forms},
\label{eq:det2}
\\
& \det \Sigma = \det \Sigma' = +1 && \text{for real bilinear  and sesquilinear forms},
\label{eq:det3}
\\
& | \det \Sigma \, | = | \det \Sigma' \, | = 1 && \text{for complex bilinear forms} .
\label{eq:det4}
\end{align}
\end{pro}
\begin{proof}
Let  $F$ be a  nonsingular matrix.  From $\Sigma=\sigma(F^{[N]}F)$, $\Sigma'=\sigma(FF^{[N]})$ and  the fact that $\sigma$ is a generalized sign function, it follows that $|\det\Sigma\,|=|\det\Sigma'\,|=1$, which is (\ref{eq:det4}). Taking the determinant of (\ref{eq:detA},~\ref{eq:detB}) it then follows that $|\det W| = 1$, which is (\ref{eq:det2}). And in the case $\mathbb{K}=\mathbb{R}$, (\ref{eq:det2}) becomes (\ref{eq:det1}). Finally, to prove (\ref{eq:det3}) for $\Sigma$, take the determinant of  $S^2=\Sigma F^{[N]}F$ and obtain 
\begin{align}
\label{eq:36}
(\det S)^2=(\det \Sigma)(\det F^{\#})(\det F).
\end{align}
For bilinear forms, (\ref{eq:36}) becomes 
\begin{align}
(\det S)^2=(\det \Sigma) (\det F)^2.
\label{eq:det5}
\end{align}
If $\mathbb{K}=\mathbb{R}$, one must then have $\det \Sigma=+1$. For sesquilinear forms, since  $S^{[N]}=S$, then $\overline{\det S}=\det S$, so $\det S$ is real. In addition,
\begin{align}
(\det S)^2=(\det \Sigma)  (\overline{\det F})(\det F).
\end{align}
Since $(\det S)^2>0$ and $(\overline{\det F})(\det F)>0$, it follows that $\det \Sigma=+1$. A similar reasoning leads to (\ref{eq:det3}) for $\Sigma'$.

\end{proof}

As a side remark, we notice that if in Definition~\ref{defn:sign_function} the condition $|\lambda|=1$ on the eigenvalues of $\sigma(A)$ is not imposed, then for the decompositions $F=WS=S'W$ in Theorems~\ref{thm:F=WS} and~\ref{thm:F=SW} one finds $\det\Sigma>0$ and $\det\Sigma'>0$ for real bilinear and sesquilinear forms,  $\det\Sigma\ne0$  and $\det\Sigma'\ne0$ for complex bilinear forms, and $| \det W| = 1/|\sqrt{\det\Sigma}|$. We have chosen the condition $|\lambda|=1$ in Definition~\ref{defn:sign_function} so as to make the matrix $W$ analogous to a unitary matrix in terms of its eigenvalues.

\section{Right and left decompositions of rectangular matrices}
\label{sec:F=WS_MN}

\begin{thm}
\label{thm:F=WS_MN}
Let $M\in \mathbb{K}^{m\times m}$ and  $N\in \mathbb{K}^{n\times n}$ be nonsingular matrices with $m\ge n$, and let $\sigma:\mathbb{K}^{n\times n}\to\mathbb{K}^{n\times n}$ be a generalized matrix sign function.  A rectangular matrix $F\in \mathbb{K}^{m\times n}$ has a decomposition
\begin{align}
\label{eq:F=WS_MN}
F = W S,
\end{align}
where,  with $\Sigma= \sigma(F^{[M,N]}F)$, the matrix $S\in \mathbb{K}^{n \times n}$ is $r$-positive-definite, $N$-selfadjoint and $N\Sigma$-selfadjoint, and the matrix  $W\in \mathbb{K}^{m \times n}$ has $(M,N\Sigma^{-1})$-orthonormal columns and satisfies $(W^{[M,N]})^{[N,M]}=W$, if and only if $F^{[M,N]}F$ is nonsingular and $(F^{[M,N]})^{[N,M]}=F$.  When such a decomposition exists it is unique, with $S$ given by $S=(\Sigma F^{[M,N]} F)^{1/2}$ and $W=FS^{-1}$.
\end{thm}

\begin{proof} IF: Assume that $F^{[M,N]}F$ is nonsingular and $(F^{[M,N]})^{[N,M]}=F$. In the latter equation, write $(F^{[M,N]})^{[N,M]}=M^{-1}(N^{-1}F^\# M)^\# N=M^{-1}M^\# FN^{-\#} N$, so  $F=M^{-1}M^\# FN^{-\#} N$. Hence
\begin{align}
\label{eq:FMNF}
(F^{[M,N]}F)^{[N]}=& N^{-1}(N^{-1}F^\# MF)^\# N=N^{-1}F^\# M^\# FN^{-\#}N \nonumber \\
=&N^{-1}F^\# M(M^{-1}M^\# FN^{-\#}N)=F^{[M,N]}F,
\end{align}
i.e.,  $F^{[M,N]}F$ is $N$-selfadjoint besides being nonsingular. Since $\sigma$ is a generalized sign function, it follows that $\Sigma F^{[M,N]}F$ with  $\Sigma= \sigma(F^{[M,N]}F)\in\mathbb{K}^{n\times n}$ is $N$-selfadjoint, nonsingular, and has no  negative real eigenvalue. Hence the principal square root
\begin{align} 
S = \Big(\Sigma F^{[M,N]}F \Big)^{1/2}
\end{align}
is well-defined and is $r$-positive-definite and $N$-selfadjoint. 
Since both $\Sigma$ and $S$  are polynomials in $F^{[M,N]} F$,  $\Sigma$ and $S$ commute.  So
\begin{align}
S^{[N\Sigma]}=\Sigma^{-1} N^{-1} S^\# N\Sigma=\Sigma^{-1} S^{[N]} \Sigma=S.
\end{align}
That is,  $S$ is $N\Sigma$-selfadjoint. 

Let $W=F S^{-1}$.  Then $F = W S$.  To show that $W$ has $(M,N\Sigma^{-1})$-orthonormal columns, we show that it  satisfies $W^{[M,N\Sigma^{-1}]}W=I_n$. 
Substituting $W=FS^{-1}$  into $W^{[M,N\Sigma^{-1}]}W$,  and using $S^{[N]}=S$, $F^{[M,N]}F = \Sigma^{-1} S^2$ and $\Sigma S = S \Sigma$, we find

\begin{align}
\label{eq:WMNW}
W^{[M,N\Sigma^{-1}]}W=&(FS^{-1})^{[M,N\Sigma^{-1}]} FS^{-1} \nonumber\\
=&S^{-[N\Sigma^{-1}]}F^{[M,N\Sigma^{-1}]} FS^{-1} \nonumber\\
=&S^{-1}\Sigma F^{[M,N]} FS^{-1} \nonumber\\
=&I_n.
\end{align}
Furthermore,
\begin{align}
\label{eq:WMNNM}
(W^{[M,N]})^{[N,M]}=&\left( (FS^{-1})^{[M,N]} \right )^{[N,M]}\nonumber\\
=&\left( S^{-[N]} F^{[M,N]} \right )^{[N,M]}\nonumber\\
=&(F^{[M,N]})^{[N,M]} (S^{-[N]})^{[N]}\nonumber\\
=&FS^{-1}\nonumber\\
=&W.
\end{align}

Now we show that the decomposition is unique.   Assume there is another decomposition $F=\tilde{W}\tilde{S}$ where $\tilde{W}$ and $\tilde{S}$ satisfy the same conditions as $S$ and $W$ listed after~(\ref{eq:F=WS_MN}).  Since $\tilde{S}$ is both $N$-selfadjoint and $N\Sigma$-selfadjoint, then $\tilde{S}$ commutes with $\Sigma$.  Thus $F^{[M,N]}F=\tilde{S}^{[N]}\tilde{W}^{[M,N]} \tilde{W} \tilde{S} =\tilde{S} \Sigma^{-1} \tilde{S}=\Sigma^{-1}  \tilde{S}^{2}$, and $\Sigma F^{[M,N]}F=\tilde{S}^{2}$. Since  $\tilde{S}$ is $r$-positive-definite and a square root of $\Sigma F^{[M,N]}F$,  so $\tilde{S}$ is the principal square root of $\Sigma F^{[M,N]}F$, which is unique. Hence $\tilde{S}=S$. The uniqueness of $W$ follows from $\tilde{W}=W=FS^{-1}$.

ONLY IF:  If $F$ has a decomposition $F=WS$ as in the text of the theorem, then 
\begin{align}
(F^{[M,N]})^{[N,M]}=&\left((WS)^{[M,N]}\right)^{[N,M]} \nonumber\\
=&\left(S^{[N]}W^{[M,N]}\right)^{[N,M]} \nonumber \\
=&(W^{[M,N]})^{[N,M]}(S^{[N]})^{[N]} \nonumber \\
=&WS \nonumber\\
=&F. 
\end{align}
Since $S$ is assumed to be both $N$-selfadjoint and $N\Sigma$-selfadjoint, then $S$ commutes with $\Sigma$, and it follows that $\Sigma F^{[M,N]}F=S^2$. Since $S$ is $r$-positive-definite, implying that it is nonsingular,  then $\Sigma F^{[M,N]}F$ is nonsingular, and  it follows that  $F^{[M,N]}F$ is nonsingular.
\end{proof}

\begin{thm}
\label{thm:F=SW_MN}
Let $M\in \mathbb{K}^{m\times m}$ and  $N\in \mathbb{K}^{n\times n}$ be nonsingular matrices with $m \le n$, and let $\sigma:\mathbb{K}^{m\times m}\to\mathbb{K}^{m\times m}$ be a generalized matrix sign function.  A rectangular matrix $F\in \mathbb{K}^{m\times n}$ has a decomposition
\begin{align}
\label{eq:F=SW_MN}
F = S \, W,
\end{align}
where,  with $\Sigma= \sigma(FF^{[M,N]})$, the matrix $S\in \mathbb{K}^{m \times m}$ is $r$-positive-definite, $M$-selfadjoint and $M\Sigma$-selfadjoint, and the matrix  $W\in \mathbb{K}^{m \times n}$ has $(M\Sigma,N)$-orthonormal rows and satisfies $(W^{[M,N]})^{[N,M]}=W$, if and only if $FF^{[M,N]}$ is nonsingular and $(F^{[M,N]})^{[N,M]}=F$.  When such a decomposition exists it is unique, with $S$ given by $S=(\Sigma F F^{[M,N]})^{1/2}$ and $W=S^{-1}F$.
\end{thm}

\begin{proof}
The proof is similar to the proof of Theorem~\ref{thm:F=WS_MN} and we only give an outline. In a way similar to (\ref{eq:FMNF}), one can show that $FF^{[M,N]}$ is $M$-selfadjoint.  Let $\Sigma= \sigma(FF^{[M,N]})$ and $S=(\Sigma F F^{[M,N]})^{1/2}$. Then $S$ is both $M$ and $M\Sigma$-selfadjoint.   $W$ is given by $W=S^{-1}F$.  In a way similar to~(\ref{eq:WMNW},~\ref{eq:WMNNM}), one can prove that $W W^{[M\Sigma,N]}=I_m$ and  $(W^{[M,N]})^{[N,M]}=W$.  The proof of uniqueness is similar to the one of Theorem~\ref{thm:F=WS_MN}.
\end{proof}

In Theorem~\ref{thm:F=WS_MN} and~\ref{thm:F=SW_MN}, if $M$ and $N$ form an orthosymmetric  pair, then  the conditions $(F^{[M, N]})^{[N,M]}=F$ and $(W^{[M,N]})^{[N,M]}=W$ are always satisfied (Definition 3.2,  Lemma 3.3 in~\cite{Higham2010}).

Notice that when $m=n$,  the matrix $F$ is a square matrix.  In this case, if $F^{[M,N]}F$ is nonsingular, then also $F$ and $FF^{[M,N]}$ are nonsingular.   By Theorems~\ref{thm:F=WS_MN} and~\ref{thm:F=SW_MN},  $F=WS=S^\prime W^\prime$, where $W$ and $S$ satisfy the conditions in Theorem~\ref{thm:F=WS_MN} and $W'$ and $S'$ those in Theorem~\ref{thm:F=SW_MN}.  The following theorem  shows that  $W^\prime=W$ and $S^\prime=WSW^{-1}$.
\begin{thm}
\label{thm:F=WS=SW_MN}
Let $M\in \mathbb{K}^{n\times n}$ and  $N\in \mathbb{K}^{n\times n}$ be nonsingular matrices, and let $\sigma:\mathbb{K}^{n\times n}\to\mathbb{K}^{n\times n}$ be a generalized matrix sign function.     Then any nonsingular matrix $F \in \mathbb{K}^{n \times n}$  such that $(F^{[M,N]})^{[N,M]}=F$ can be factorized uniquely as
\begin{align}
F = W S=S^\prime W,
\end{align}
where,  with $\Sigma= \sigma(F^{[M,N]}F)$ and $\Sigma^\prime= \sigma(FF^{[M,N]})$, the matrix  $S\in \mathbb{K}^{n \times n}$ is $r$-positive-definite, $N$-selfadjoint and $N\Sigma$-selfadjoint,  the matrix  $S^\prime \in \mathbb{K}^{n \times n}$ is $r$-positive-definite, $M$-selfadjoint and $M\Sigma^\prime$-selfadjoint,  and the matrix $W\in \mathbb{K}^{n \times n}$  is $(M,N\Sigma^{-1})$-unitary, $(M\Sigma^\prime, N)$-unitary and satisfies  $(W^{[M,N]})^{[N,M]}=W$.  The matrices $S$, $S'$, and $W$ are given by $S=(\Sigma F^{[M,N]} F)^{1/2}$, $S'=(\Sigma'F  F^{[M,N]})^{1/2}$, and $W=FS^{-1} = S^{\prime-1} F$.
\end{thm}

\begin{proof}
Let $\Sigma=\sigma(F^{[M,N]}F)$ and $\Sigma^\prime=\sigma(FF^{[M,N]})$. By Theorem~\ref{thm:F=WS_MN},  $F=WS$ with $S=(\Sigma F^{[M,N]} F)^{1/2}$ and $W=FS^{-1}$.    
Since for nonsingular $F$,
\begin{align}
FF^{[M,N]}=&F (F^{[M,N]}F)F^{-1}=WS \Sigma ^{-1} S^2 (WS)^{-1} \nonumber\\
=&W \Sigma ^{-1} S^2 W^{-1}=W (F^{[M,N]}F)W^{-1},
\end{align}
then $\Sigma^\prime=W\Sigma W^{-1}$ and $S^\prime=(\Sigma^\prime F F^{[M,N]})^{1/2}=W(\Sigma F^{[M,N]}F)^{1/2}W^{-1}=WSW^{-1}$.  So $F=WS=WSW^{-1}W=S^\prime W$. By Theorems~\ref{thm:F=WS_MN} and~\ref{thm:F=SW_MN}.  $W$ is both $(M, N\Sigma^{-1})$-unitary and $(M\Sigma^\prime, N)$-unitary. 
\end{proof}

When $M=N$, Theorem~\ref{thm:F=WS=SW_MN} reduces to Theorems~\ref{thm:F=WS} and~\ref{thm:F=SW}.

\section{Comments and examples}
\label{sec:exmp}

In this section,  we make some comments on the previous theorems and give some examples.  

\subsection{Comments and examples on Theorems~\ref{thm:F=WS} and~\ref{thm:F=SW}}

In Theorem~\ref{thm:F=WS}, when $\mathbb{K}=\mathbb{C}$, the same matrix $N$ can define a complex bilinear product or a sesquilinear product, and in general the corresponding $F=WS$ decompositions are different.
\begin{exmp}
\label{exmp:a}
Let  the generalized sign function  be defined by~(\ref{eq:Sign_1}), and let $F=\begin{pmatrix}-1+2 i\end{pmatrix} \in \mathbb{C}^{1\times1}$.

Let the matrix $N=\begin{pmatrix} 1\end{pmatrix}$ define a complex  bilinear product. Then $(F^{[N]})^{[N]}=F$ and one computes $F^{[N]}F=N^{-1}F^TNF=\begin{pmatrix}-3-4i\end{pmatrix}$,   $\Sigma=\begin{pmatrix}-1\end{pmatrix}$, $S=\begin{pmatrix}2+i\end{pmatrix}$, and $W=(i)$.  Therefore $F=WS=\begin{pmatrix}i\end{pmatrix}\begin{pmatrix}2+i\end{pmatrix}$.

Let the matrix $N=\begin{pmatrix} 1\end{pmatrix}$ define a sesquilinear product. Then $(F^{[N]})^{[N]}=F$ and one computes $F^{[N]}F=N^{-1}\bar{F}^TNF
=\begin{pmatrix}5\end{pmatrix}$,  $\Sigma=\begin{pmatrix}1\end{pmatrix}$, $S=\begin{pmatrix}\sqrt{5}\end{pmatrix}$ and $W=\begin{pmatrix}\frac{-1+2i}{\sqrt{5}}\end{pmatrix}$.
Therefore $F=WS=\begin{pmatrix}\frac{-1+2i}{\sqrt{5}}\end{pmatrix}\begin{pmatrix}\sqrt{5}\end{pmatrix}$.
\end{exmp}

Different generalized sign functions give rise to different $F=WS$ decompositions in Theorem~\ref{thm:F=WS}.

\begin{exmp}
In  Example~\ref{exmp:a} with $N=(1)$ defining a complex bilinear product, replace the generalized sign function~(\ref{eq:Sign_1}) by~(\ref{eq:Sign})  and~(\ref{eq:Sign_3}), respectively. 

With the generalized sign function in~(\ref{eq:Sign}), from $F^{[N]}F
=\begin{pmatrix}-3-4i\end{pmatrix}$ one finds $\Sigma=\begin{pmatrix}1\end{pmatrix}$ and  $S=(\Sigma F^{[N]}F)^{1/2}=\begin{pmatrix}1-2i\end{pmatrix}$, therefore $F=WS=\begin{pmatrix}-1\end{pmatrix}\begin{pmatrix}1-2i\end{pmatrix}$. 

With the generalized sign function in~(\ref{eq:Sign_3}), from $F^{[N]}F
=\begin{pmatrix}-3-4i\end{pmatrix}$ one finds  $\Sigma=\begin{pmatrix}\frac{-3+4i}{5}\end{pmatrix}$ and $S=\begin{pmatrix}\sqrt{5}\end{pmatrix}$, therefore $F=WS=\begin{pmatrix}\frac{-1+2i}{\sqrt{5}}\end{pmatrix}\begin{pmatrix}\sqrt{5}\end{pmatrix}$.
\end{exmp}

In general, the decomposition $F=WS$ in Theorem~\ref{thm:F=WS} of a given matrix $F$ differs from other kinds of polar decompositions in the literature, such as the generalized polar decomposition, the $H$-polar decomposition and polar decompositions defined for particular product matrices. 

The decomposition $F=WS$  is  related to the generalized polar decomposition in the sense that the $N$-selfadjoint matrix  $S$ is $r$-positive-definite.  In particular, for the generalized sign function  defined by~(\ref{eq:Sign}),  if $F^{[N]}F$ has no  negative real eigenvalues, then $\Sigma=I_n$ and   $W$ is $N$-unitary, therefore,  $F=WS$ is  a generalized polar decomposition. 
 
An $H$-polar decomposition is defined for  an indefinite inner product  specified by a selfadjoint matrix $H\in \mathbb{K}^{n\times n}$. In general, the decomposition  $F=WS$ of a matrix $F$ for a selfadjoint matrix $N=H$ is different from an $H$-polar decomposition $F=L_HS_H$ of the same matrix $F$.   

\begin{exmp} Let  the generalized sign function be defined by~(\ref{eq:Sign}).
Let the selfadjoint matrix $H=\begin{pmatrix}
1&\\
&-4
\end{pmatrix}$ define a real bilinear product.  
Let $F=\begin{pmatrix}
&4\\
1&
\end{pmatrix}$, which satisfies  $(F^{[H]})^{[H]}=F$. 

The decomposition $F=WS$ in Theorem~\ref{thm:F=WS} is computed as follows.
$
 F^{[H]}F=\begin{pmatrix}
 -4&\\
 &-4
 \end{pmatrix}$,
$ \Sigma=\begin{pmatrix}
 -1&\\
 &-1
 \end{pmatrix}$, 
$S=\begin{pmatrix}
2& \\
&2
\end{pmatrix}
$
and
$W = \begin{pmatrix}
&2\\
\frac{1}{2}&
\end{pmatrix}
$.
Therefore
$
F=WS=\begin{pmatrix}
&2\\
\frac{1}{2}&
\end{pmatrix}
\begin{pmatrix}
2&\\
&2
\end{pmatrix}.
$
Since $W$ is not $H$-unitary,  this $F=WS$ decomposition is not an $H$-polar decomposition.  

One possible $H$-polar decomposition is
$
 F=L_HS_H=\begin{pmatrix}
 -1&\\&1
 \end{pmatrix}
 \begin{pmatrix}
 &-4\\
 1&
 \end{pmatrix}.
$
\end{exmp}

The decomposition $F=WS$ in Theorem~\ref{thm:F=WS} is also in general different from other kinds of polar decompositions defined for scalar products given by particular matrices such as  $N=I_n$,  $N=Z$, $N=J$, or $N=D$ (see, e.g.,~\cite{Cardoso2002, Mackey2006}). Here
 \begin{align}
 Z=\begin{pmatrix}
 &&1\\
 &\adots&\\
 1&&
 \end{pmatrix}, \quad
 J=\begin{pmatrix}
 &I_m\\
 -I_m 
 \end{pmatrix}, \quad
 D=\begin{pmatrix}
 I_p&\\
 &-I_q
 \end{pmatrix}.
 \end{align}

\begin{exmp} Let  the generalized sign function  be defined by~(\ref{eq:Sign}).
Let the matrix $J=\begin{pmatrix}
&1\\
-1&
\end{pmatrix}$ define a sesquilinear  product. 
Let  $F=\begin{pmatrix}
i&\\
&-4i
\end{pmatrix}$, which satisfies $(F^{[J]})^{[J]}=F$.

Table 2.1 in~\cite{Mackey2006} defines a polar decomposition $F=L_JS_J$,  where $L_J$ is $J$-conjugate symplectic and $S_J$ is $J$-skew-Hermitian. One possible such decomposition is $F=L_JS_J=\begin{pmatrix}\frac{1}{2}&\\&2\end{pmatrix}\begin{pmatrix}2i&\\&-2i\end{pmatrix}$.

The $F=WS$ decomposition in Theorem~\ref{thm:F=WS} gives instead $F^{[J]}F=\begin{pmatrix}-4&\\&-4\end{pmatrix}$, $\Sigma=\begin{pmatrix}-1&\\&-1\end{pmatrix}$, $S=\begin{pmatrix}2&\\&2\end{pmatrix}$, and $W=\begin{pmatrix}\frac{i}{2}&\\&-2i\end{pmatrix}$, therefore  $F=WS=\begin{pmatrix}\frac{i}{2}&\\&-2i\end{pmatrix}\begin{pmatrix}2&\\&2\end{pmatrix}$.
\end{exmp}

Finally, we give an example of the decompositions $F=WS$ in Theorem~\ref{thm:F=WS} and $F=S'W$ in Theorem~\ref{thm:F=SW}  when $N$ is not  orthosymmetric.  

\begin{exmp}Let  the generalized sign function  be defined by~(\ref{eq:Sign_3}).
Let the nonorthosymmetric matrix $N=\begin{pmatrix} &1\\2i &\end{pmatrix}$ define a sesquilinear product.
Let  $F=\begin{pmatrix}-1&\\&4i\end{pmatrix}$.
One computes
$F^{[N]}F=\begin{pmatrix}4i&\\&-4i\end{pmatrix}$, $\Sigma=\begin{pmatrix}-i&\\&i\end{pmatrix}$ and $S=\begin{pmatrix}2&\\&2\end{pmatrix}$.  Therefore $F=WS=\begin{pmatrix}-\frac{1}{2}&\\&2i\end{pmatrix}\begin{pmatrix}2&\\&2\end{pmatrix}$ and $F=S^\prime W=\begin{pmatrix}2&\\&2\end{pmatrix}\begin{pmatrix}-\frac{1}{2}&\\&2i\end{pmatrix}$.   In this particular example, $S=S^\prime$.
\end{exmp}

\subsection{Comments and examples on Theorems~\ref{thm:F=WS_MN}-\ref{thm:F=WS=SW_MN}}

In Theorem~\ref{thm:F=WS_MN}, the $F=WS$ decomposition depends on the nature of the scalar product defined by the matrices $M$ and $N$ (bilinear or sesquilinear product).
      
\begin{exmp}
Let  the generalized sign function  be defined by~(\ref{eq:Sign_3}).  Let $M=\begin{pmatrix}&1 \\i & \end{pmatrix}$,  $N=\begin{pmatrix}1+i\end{pmatrix}$, and $F=\begin{pmatrix}1\\-4i\end{pmatrix}$.

If $M$ and $N$ define complex bilinear products, $(F^{[M,N]})^{[N,M]}=\begin{pmatrix}-i\\4\end{pmatrix}\ne F$,  and $F$ does not have an $F=WS$ decomposition. 

If $M$ and $N$ define sesquilinear products, then $(F^{[M,N]})^{[N,M]}=F$ and one computes $F^{[M,N]}F=\begin{pmatrix}-4\end{pmatrix}$,  $\Sigma=\begin{pmatrix}-1\end{pmatrix}$,  $S=\begin{pmatrix}2\end{pmatrix}$, and therefore $F=WS=\begin{pmatrix}\frac{1}{2}\\ -2i\end{pmatrix}\begin{pmatrix}2\end{pmatrix}$.
\end{exmp}

Theorems~\ref{thm:F=WS_MN} is related to the canonical generalized polar decomposition in~\cite{Higham2010} in the sense that the $N$-selfadjoint matrix  $S$ is $r$-positive-definite.  In particular,  let  the generalized sign function  be defined by~(\ref{eq:Sign}), and let  $M$ and $N$  form an orthosymmetric pair. If $F^{[M,N]}F$ has no  negative real eigenvalues, then $\Sigma=I_n$ and $W$ is $(M,N)$-unitary. Therefore,  in this case, $F=WS$ is  a canonical  generalized polar decomposition.

In general,  $M$ and $N$ do not need to form an orthosymmetric pair.   A more general condition is given by  $(F^{[M,N]})^{[N,M]}=F$. 

\begin{exmp}
Let  the generalized sign function  be defined by~(\ref{eq:Sign_3}).  Let  the matrices $M=\begin{pmatrix}&1 \\i & \end{pmatrix}$ and  $N=\begin{pmatrix}& i\\-1&\end{pmatrix}$ define two  complex bilinear  products.   Let $F=\begin{pmatrix}-1&\\&4\end{pmatrix}$, which satisfies  $(F^{[M,N]})^{[N,M]}=F$. One computes $F^{[M,N]}F=\begin{pmatrix}4i&\\&4i\end{pmatrix}$,  $\Sigma=\begin{pmatrix}-i&\\&-i\end{pmatrix}$, $S=\begin{pmatrix}2&\\&2\end{pmatrix}$, and $W=\begin{pmatrix}-\frac{1}{2}&\\ &2\end{pmatrix}$. Therefore $F=WS=\begin{pmatrix}-\frac{1}{2}&\\ &2\end{pmatrix}\begin{pmatrix}2&\\&2\end{pmatrix}$.
\end{exmp}

\begin{exmp}
Let  the generalized sign function  be defined by~(\ref{eq:Sign}).   Let the matrices $M=\begin{pmatrix}4i & \\ &1 \end{pmatrix}$ and  $N=\begin{pmatrix}1&\\ &-2i\end{pmatrix}$ define two  sesquilinear  products. Let  $F=\begin{pmatrix}&1\\3i&\end{pmatrix}$, which satisfies  $(F^{[M,N]})^{[N,M]}=F$.  One computes  $F^{[M,N]}F=\begin{pmatrix}9&\\&-2\end{pmatrix}$,  $\Sigma=\begin{pmatrix}1&\\&-1\end{pmatrix}$, $S=\begin{pmatrix}3&\\&\sqrt{2}\end{pmatrix}$, and $W=\begin{pmatrix}& \frac{1}{\sqrt{2}}\\i&\end{pmatrix}$. Therefore $F=WS=\begin{pmatrix}& \frac{1}{\sqrt{2}}\\i&\end{pmatrix}\begin{pmatrix}3&\\&\sqrt{2}\end{pmatrix}$. Similarly one finds $F=S^\prime W=\begin{pmatrix}\sqrt{2}&\\ &3\end{pmatrix}\begin{pmatrix}&\frac{1}{\sqrt{2}}\\i&\end{pmatrix}$.
\end{exmp}

\section*{Acknowledgements}

This research was supported by the University of Utah.

%=======================================================================


%% References


%\bibliographystyle{model1b-num-names}

%% not want to use model1b-num-names.bst.


\end{document}